\documentclass[12pt]{article}
\usepackage{ifthen}
\usepackage{amsmath}
\usepackage{amsthm}
\usepackage{tikz}

\newtheorem{lemma}{Lemma}
\newtheorem{theorem}{Theorem}

\newboolean{PrintVersion}
\setboolean{PrintVersion}{false}

\usepackage{amsmath,amssymb,amstext} %

\usepackage{amsmath, amsfonts, amsthm, amssymb}  %
\usepackage{enumerate}
\usepackage[all]{xy}
\usepackage{wrapfig}
\usepackage{fancyvrb}
\usepackage{listings}

\usepackage{centernot}
\usepackage{mathtools}

\renewcommand{\H}{{\mathcal H}}
\renewcommand{\L}{{\mathcal L}}

\renewcommand{\i}{^{-1}}

\DeclareMathOperator{\N}{N}

\newcommand{\A}{{\mathbb A}}
\newcommand{\p}{{\mathbb{P}^1(\mathbb{C})}}

\usepackage[T1]{fontenc}
\usepackage{textcomp}
\usepackage{amssymb}
\usepackage{theoremref}

\usepackage[backend=bibtex]{biblatex}
\usepackage{import}
\usepackage{xifthen}
\usepackage{pdfpages}
\usepackage{transparent}

\title{Incidence Bounds on Edge Partitions of $K_n$}
\author{Andean E. Medjedovic}
\date{}

\begin{document}

\maketitle

\begin{abstract}
	We solve a problem conjectured by Cheriyan, giving sharp bounds for incidence
	of certain edge partitions of the connected graph on $n$-vertices.
	We briefly discuss the history of the problem and
	relation to node connectivity
	of strongly regular graphs.
	We show that the bound cannot be made sharper.

\end{abstract}
\section{Introduction}
Several characterizations of edge-connectivity are known within the literature. See, for instance,  \cite{Nash_williams_1960},  \cite{Kir_ly_2006}, or  \cite{Thomassen_2015}. In recent years there has been a push to understand an analogous notion of connectivity for nodes  \cite{Kir_ly_2008}. One earlier result
within this area is found in \cite{cheriyan18}.

\begin{theorem}\label{thm2}
Let $G$ be a $2k$-regular $2k$-node connected graph such that for every set of nodes $S$ with
$1 \leq |S|  \leq \frac{|V(G)|}{2}$ we have $|N_G(S)| > \min \{ k^2 - 1, (k - 1)(|S| + 1) \}$. Then every Eulerian
orientation of $G$ is strongly $k$-node connected.
\end{theorem}

Here $|N_G(S)|$ is, of course, the number of nodes in $G$ adjacent to $S$.
This theorem is then used to prove the following within the same paper.

\begin{theorem}
	Every Eulerian orientation of the hypercube of degree $2k$ is strongly $k$-node connected.
\end{theorem}

In an effort to extend this to a larger family of strongly regular graphs the
Johnson $J(n,2)$ (the line graph of $K_n$) was considered. It was remarked that this problem is equivalent to proving the main theorem within this work. Since then, the analogue for theorem \ref{thm2} on $J(n,2)$ was proven in \cite{Hoersch2019} but the incidence bound remained an open question.

In this paper we provide an alternate proof to the one found in \cite{Hoersch2019} of the parallel result for the line graph of the connected graph $K_n$.
In doing so we prove certain incidence bounds for edge partitions of $K_n$.

The author would like to thank Joseph Cheriyan for his comments and encouragements.
Without him this paper would not have been possible.

\section{The main result}
Let $K_n$ be the connected graph on $n$ vertices. Partition the edges of $K_n$ into $3$ sets $S$,$T$, and $Z$ with $|Z| = n -3$.
We prove that the incidence of $S$ and $T$ is at least the minimum of the incidence of either $Z$ and $T$ or $Z$ and $S$.

Label the vertices $v_1, \ldots, v_n$ and let the degree in $S,T,Z$ of $v_i$ be $s_i,t_i,z_i$, respectfully. We prove

\begin{theorem}{\label{t1}}
	\[
	\sum_{i=1}^n s_it_i \ge \min \{\sum_{i=1}^n z_i t_i , \sum_{i=1}^n z_i s_i\}\\
	.\]
\end{theorem}

We begin with an elementary observation, since the degree of each node is $n-1$,it follows that $s_i+t_i+z_i = n-1$.
Suppose $s_1, \ldots, s_p$ are the nodes with degree in $S$ equal to $0$, let $t_{p+1},
,\ldots, t_{p+q}$ be the nodes with degree in $T$ being $0$. We let $P = \{v_1, \ldots, v_p\}$,
$Q = \{v_{p+1}, \ldots v_{p+q}\}$, $R = V\setminus(P\cup Q)$.

Assume for contradiction that $\sum_{i=1}^n s_i t_i < \sum_{i=1}^n z_i t_i$ and $\sum_{i=1}^n s_i t_i < \sum_{i=1}^n z_i s_i$. We first prove a few quick lemmas:

\begin{lemma}{\label{l1}}
	We have:
	$$2(n-1)(n-3) -\sum_{i=1}^n z_i^2 = (n-1)\sum_{i=1}^n z_i -\sum_{i=1}^n z_i^2 > \sum_{i=1}^n (n-1)t_i  - \sum_{i=1}^n t_i^2$$ \label{eq1}

	$$2(n-1)(n-3) -\sum_{i=1}^n z_i^2 = (n-1)\sum_{i=1}^n z_i -\sum_{i=1}^n z_i^2 > \sum_{i=1}^n (n-1)s_i  - \sum_{i=1}^n s_i^2$$ \label{eq2}
	And:
	\begin{equation}
		\frac{\sum_{i=1}^n z_i^2}{2} + \sum_{i=1}^n s_i t_i < (n-1)(n-3)
	\end{equation}\\
\end{lemma}
\begin{proof}
From
	$$\sum_{i=1}^n s_it_i < \sum_{i=1}^n z_it_i$$
	Write $t_i = n-1-s_i-z_i$ to get

	$$(n-1)\sum_{i=1}^n z_i -\sum_{i=1}^n z_i^2 > \sum_{i=1}^n (n-1)s_i  - \sum_{i=1}^n s_i^2$$
	And by symmetry we have the other inequality. For the last inequality simply sum the $2$ incidence inequalities (note that $\sum z_i = 2(n-3)$):

	$$ 2 \sum_{i=1}^n s_i t_i < \sum_{i=1}^n z_i (t_i+s_i)$$
	$$ 2\sum_{i=1}^n s_i t_i -\sum_{i=1}^n z_i^2 < \sum_{i=1}^n z_i(n-1) = 2(n-1)(n-3)$$

	And divide through by $2$.
\end{proof}

\begin{lemma}{\label{p+q}}
	\[
	p+q \le 2\sqrt{n-3}
	.\]
\end{lemma}
\begin{proof}
We have exactly $n-q$ non-zero degree in $S$ vertices so for each $s_i$ we have $s_i \leq n-1-q$ and similarily $t_i \leq n-1-p$. Then for $v_i \in P,Q$ we must have $z_i \geq p,q$.
Using this we bound the sum $p+q$, since $p^2 + q^2 \leq \sum_{P}z_i +\sum_{Q}z_i \leq \sum_{i=1}^n z_i = 2(n-3)$. The maximum over $p$ and $q$ is achieved when $1$ term dominates, so $p+q$ is at most $2\sqrt{n-3}$.
\end{proof}

\begin{lemma}{\label{z1}}
	$$\sum_{i=1}^n \frac{z_i^2}{2} +3 \geq \sum_R z_i$$
\end{lemma}

\begin{proof}
	The sum of squares is smallest when they are evenly distributed, in this case this is achieved when $z_i =2 $ for $n-6$ of the vertices and $1$ for the remaining $6$. We also know $\sum_R z_i \leq \sum_{i=1}^n z_i = 2(n-3)$:
	$$\sum_{i=1}^n \frac{z_i^2}{2} \geq \frac{2^2}{2}(n-6) + \frac{6}{2} +3 = 2(n-3) \geq \sum_R z_i$$
\end{proof}

Now for $p + q \leq 2$, using \ref{l1} and \ref{z1}
\begin{equation}
	\begin{split}
		\sum_{i=1}^n s_it_i + \sum_{i=1}^n \frac{z_i^2}{2} & = \sum_R s_it_i + \sum_{i=1}^n \frac{z_i^2}{2}\\
							& \geq \sum_R (n-2-z_i) + \sum_{i=1}^n \frac{z_i^2}{2}\\
							& = (n-p-q)(n-2) - \sum_R z_i +\sum_{i=1}^n \frac{z_i^2}{2}\\
							& \geq (n-p-q)(n-2) -3 \geq (n-2)(n-2) \geq (n-1)(n-3)
	\end{split}
\end{equation}

Contradiction.

We can WLOG assume $p \geq q$. Assume $p\geq 4$. Then by summing the first two lines of \ref{l1}
\begin{equation}
\begin{split}
4(n-1)(n-3) - 2 \sum_{i=1}^n z_i^2 & > (n-1)\sum_{i=1}^n (n-1-z_i) - \sum_{i=1}^n(t_i^2+s_i^2) \\
					     & = n(n-1)^2 - 2(n-1)(n-3) -\sum_{P,Q} (t_i^2 +s_i^2) -\sum_R(t_i^2+s_i^2) \\
					     & > n(n-1)^2 -2(n-1)(n-3) -\sum_{P,Q} (n-1-z_i)^2\\
					     & \hspace{5mm}- (n-p-q)((n-1-p)^2+p^2) \\
					     & = n(n-1)^2 -2(n-1)(n-3) -\sum_{P,Q}z_i^2 +2(n-1)(\sum_{P,Q} z_i)\\
					     & \hspace{5mm}-(p+q)(n-1)^2 -(n-p-q)((n-1-p)^2 +p^2)
	\end{split}
\end{equation}

Rearranging and simplifying terms now gives us:

$$6(n-1)(n-3) > 2\sum_{i=1}^n z_i^2 +2(n-1)(\sum_{P,Q} z_i) -\sum_{P,Q}z_i^2 +2pn^2 +2p(-1-2p-q)n +2p(p^2+q+pq) $$
The goal is to prove the RHS is larger than the LHS for contradiction. Notice:
\[
	2\sum_{i=1}^n z_i^2 +2(n-1)(\sum_{P,Q} z_i) -\sum_{P,Q}z_i^2 \ge 2 \sum_{i=1}^n z_i + \sum_{P,Q}\left( 2(n-1)-z_i \right) z_i \ge 4(n-3)+(n+1)q
.\]
The last $(n+1)q$ is from at least $z_i = 1$ for all $i \in Q$, otherwise $z_i = 0$ implies there is some vertex with all edges in $S$ meaning $p=0$.
And bounding $2(n-1) -z_i \ge n+1$.\\

Substituting the above in and collecting $n$ yields
\[
	0 > n^2(2p-6)+n(2p(2p-q-1)+q+28)+ 2p(p^2+pq+q)+q-30
.\]

We assumed that $p$ is at least $4$. The worst case for the above is when $q = p$ with the largest zero of the RHS occurring at
\begin{multline}
	n = \frac{1}{4(p - 3)}\bigg(\sqrt{-32 p^3 (p - 1) + 4 p^2 (p^2 + 12 p + 57) - 4 p (p^2 + 19 p - 32) + p^2 + 80 p + 64}\\-4 p^2  + 2 p (p+ 1) - p - 28\bigg)
\end{multline}
Recall the bound $p+q < 2\sqrt{n-3} $. This gives $\left( \frac{p}{2} \right)^2 +3 < n$. Substitute this for $n$ in the above. What one obtains is a polynomial that is
positive for all $p\ge 4$. Our contradictive assumption was that this expression is negative, contradiction.\\

We have only a few cases left to consider: $(p,q)\in \{(2,1),(2,2),(3,0),(3,1),(3,2),(3,3)\}$.
One can verify by computer that the conjecture is true for $n\leq 15$ for the given $(p,q)$. For larger $n$ we use the following technique.

By symmetry we can assume that $p \ge q$ and that $p$ is at least $2$. Consider the $2$ vertices, $v_1,v_2$ in $P$. Let $P_2$ be the subset of vertices in $R$ connected to both $v_1$ and $v_2$ via an edge in $T$. Use $\sigma$ to denote $|P_2|$. We must then have $z_1+z_2 = 2(n-1)-t_1-t_2$ by degree considerations on $P$.
Also, $\sigma \ge t_1 + t_2-2p+3 -(n-q-p-2)= t_1+t_2-p+q-n+5$ by pigeonhole principle. This comes from the realization that in the worst case we have $t_1+t_2-2p+3$ ($-2p+3$ comes from edges contained in $P$) for the case where   edges going to the $n-p-q-2$ vertices in $R$.
A double count occurs at least $t_1 + t_2-2p+3 -(n-q-p-2)$ many times.

	We intend to show
	\[
		\frac{1}{2}\sum_{i=1}^n z_i^2 + \sum_{i}s_it_i \ge (n-1)(n-3)
	.\]

	Indeed
\begin{equation}
	\begin{split}
		\frac{1}{2}\sum_{i} z_i^2 + \sum_{i}s_it_i &= 	\frac{1}{2}\sum_{i}z_i^2 + \sum_{P_2}s_it_i+\sum_{R \setminus P_2}s_it_i\\
							 &= 	\frac{1}{2}\sum_{i=1}^n z_i^2 + \sum_{P_2}2(n-3-z_i) + \sum_{R \setminus P_2} (n-2-z_i)\\
							 &\ge	 \frac{1}{2}\sum_{i=1}^n z_i^2 - \sum_{R} z_i - \sum_{P_2}z_i+ 2(n-3)\sigma + (n-p-q-\sigma)(n-2)\\
							 &= C_z + n^2 + n(\sigma-p-q-2) -4\sigma +2(p+q).
	\end{split}
\end{equation}

Where $C_z := \frac{1}{2}\sum_{i=1}^n z_i^2 - \sum_{R} z_i - \sum_{P_2}z_i$. Comparing the above to $(n-1)(n-3)$ gives a difference of
 \[
	 n(\sigma-p-q+2) + C_z +2(p+q) - 3 - 4\sigma
.\]
That is, we wish to show
\begin{equation}\label{frac}
	n > \frac{ \left(4\sigma + 3 -2(p+q) -C_z \right)}{\sigma -p -q +2}
\end{equation}
for a contradiction.

\begin{lemma}
$$\sigma -p -q +2 \ge 1$$
\end{lemma}
\begin{proof}
From the above we have $\sigma \ge t_1+t_2-p+q-n+5$ or
\[
\sigma + z_1+z_2 \ge n+3-p+q
.\]
And since $z_1+z_2$ is at most $n-2$ (we can have $1$ adjacent edge among the $n-3$ in $Z$)
\[
\sigma -p-q+2 \ge n+3-2p+2-z_1-z_2 \ge 7-2p \ge 1
.\]
\end{proof}

The penultimate step is using this easy bound on $C_z$

\begin{lemma}
	\[
		-C_z = \sum_{P_2}2z_i -\frac{1}{2}z_i^2 + \sum_{R\setminus P_2}z_i-\frac{1}{2}z_i^2 \le 2\sigma + \frac{1}{2}(n-p-q-\sigma)
	.\]
\end{lemma}

	We can finally simplify the fraction on (\ref{frac}) to

$$ \frac{ \left(4\sigma + 3 -2(p+q) -C_z \right)}{\sigma -p -q +2} \le \frac{1}{2}n +\frac{15}{2}$$
by finding common denominators and reducing.\\

So we are left with $n > \frac{1}{2}n+\frac{15}{2}$, the theorem is true for $n > 15$, as required.

To summarize, we are left with
\begin{theorem}\label{thm1}
	Suppose $S,T,Z$ is an edge partition of $K_n$ with $|Z| = n-3$.
	Then
	\begin{equation}
		\sum_{i=1}^n s_it_i \ge \min \{\sum_{i=1}^n z_i t_i , \sum_{i=1}^n z_i s_i\}
	\end{equation}
	Where $s_i, t_i, z_i$ are the respective degrees of vertex $v_i$ in $S,T,Z$.
\end{theorem}

\section{Sharp Bounds}

We now demonstrate that the bound obtained on the incidence is in some sense the best possible.
Suppose we strengthened the condition on the size of $Z$ to be $n-2$ instead of $n-3$. In this case we can construct a family of counterexamples.

Let $n > 5$ and consider the edge partition

$$	S = \{v_1,v_2\},\{v_2,v_3\} $$
$$Z = \{v_1,v_3\},\{v_2,v_i\} \hspace{5mm} \forall i \in \{4,\ldots,n\}$$
$$T = \text{the remaining nodes.}$$

The incidence of $S$ and $T$ is then $2(n-3)$ while the incidence of $S$ and $Z$ is $2(n-3)+2$. The incidence between $T$ and $Z$ is at least $(n-3)(n-1)$. Therefore,
the bound from \ref{thm1} fails.\\

We've added an illustration of the case $n=5$ below.

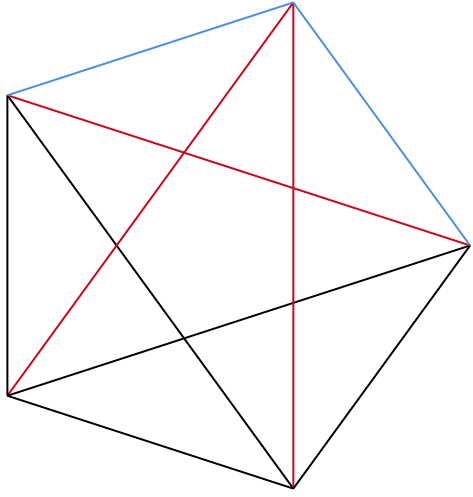
\begin{center}
\tikzset{every picture/.style={line width=0.75pt}} %

\begin{tikzpicture}[x=0.75pt,y=0.75pt,yscale=-1,xscale=1]
\draw  [draw opacity=0] (443,137) -- (353.86,259.69) -- (209.64,212.82) -- (209.64,61.18) -- (353.86,14.31) -- cycle ;
\draw    (209.64,61.18) -- (353.86,259.69) ;
\draw [color={rgb, 255:red, 208; green, 2; blue, 27 }  ,draw opacity=1 ]   (353.86,14.31) -- (209.64,212.82) ;
\draw    (443,137) -- (209.64,212.82) ;
\draw [color={rgb, 255:red, 208; green, 2; blue, 27 }  ,draw opacity=1 ]   (353.86,14.31) -- (353.86,259.69) ;
\draw [color={rgb, 255:red, 208; green, 2; blue, 27 }  ,draw opacity=1 ]   (209.64,61.18) -- (443,137) ;
\draw    (209.64,61.18) -- (209.64,212.82) ;
\draw [color={rgb, 255:red, 74; green, 144; blue, 226 }  ,draw opacity=1 ]   (209.64,61.18) -- (353.86,14.31) ;
\draw [color={rgb, 255:red, 74; green, 144; blue, 226 }  ,draw opacity=1 ]   (353.86,14.31) -- (443,137) ;
\draw    (353.86,259.69) -- (209.64,212.82) ;
\draw    (443,137) -- (353.86,259.69) ;

\end{tikzpicture}
\end{center}
The blue edges are in $S$, the black, $T$, and red, $Z$.

\section{Eulerian Orientations of the Line Graph of $K_n$}

Recall theorem \ref{thm2}. Following the logic of the proof as in \cite{cheriyan18} we see that
\[
	\sum s_i t_i \geq \min \{ \sum_{i=1}^{n} z_i t_i , \sum_{i=1}^{n} z_i s_i \}
.\]
implies, for $k = n-2$,
\[
	|N_G(\hat{S})| \geq \min \{ k|\hat{Z}|, |\hat{S} | |\hat{Z}| \} > \min \{ k^2 - 1 , (k-1)(|\hat{S}| +1) \}
.\]
where $\hat{S} , \hat{Z}$ are induced vertex subsets of $J(n,2)$ coming from edges in $K_n$ after taking the line graph. As a consequence of the incidence bounds,

\begin{theorem}
	Every Eulerian orientation of the line graph of $K_n$ of degree $2(n-2)$ is strongly $(n-2)$-node connected.
\end{theorem}

\newpage

\printbibliography

\end{document}